\documentclass{amsart} 
\usepackage{amssymb}
\usepackage{amsmath}
\usepackage{amsfonts}

\begin{document}
\newtheorem{theo}{Theorem}[section]
\newtheorem{prop}[theo]{Proposition}
\newtheorem{lemma}[theo]{Lemma}
\newtheorem{exam}[theo]{Example}
\newtheorem{coro}[theo]{Corollary}
\theoremstyle{definition}
\newtheorem{defi}[theo]{Definition}
\newtheorem{rem}[theo]{Remark}


\newcommand{\Bb}{{\bf B}}
\newcommand{\Nb}{{\bf N}}
\newcommand{\Qb}{{\bf Q}}
\newcommand{\Rb}{{\bf R}}
\newcommand{\Zb}{{\bf Z}}
\newcommand{\Ac}{{\mathcal A}}
\newcommand{\Bc}{{\mathcal B}}
\newcommand{\Cc}{{\mathcal C}}
\newcommand{\Dc}{{\mathcal D}}
\newcommand{\Fc}{{\mathcal F}}
\newcommand{\Ic}{{\mathcal I}}
\newcommand{\Jc}{{\mathcal J}}
\newcommand{\Lc}{{\mathcal L}}
\newcommand{\Oc}{{\mathcal O}}
\newcommand{\Pc}{{\mathcal P}}
\newcommand{\Sc}{{\mathcal S}}
\newcommand{\Tc}{{\mathcal T}}
\newcommand{\Uc}{{\mathcal U}}
\newcommand{\Vc}{{\mathcal V}}

\newcommand{\ax}{{\rm ax}}
\newcommand{\Acc}{{\rm Acc}}
\newcommand{\Act}{{\rm Act}}
\newcommand{\ded}{{\rm ded}}
\newcommand{\Gm}{{$\Gamma_0$}}
\newcommand{\ID}{{${\rm ID}_1^i(\Oc)$}}
\newcommand{\PA}{{\rm PA}}
\newcommand{\ACA}{{${\rm ACA}^i$}}
\newcommand{\RefP}{{${\rm Ref}^*({\rm PA}(P))$}}
\newcommand{\RefS}{{${\rm Ref}^*({\rm S}(P))$}}
\newcommand{\Rfn}{{\rm Rfn}}
\newcommand{\tar}{{\rm Tarski}}
\newcommand{\UNFA}{{${\mathcal U}({\rm NFA})$}}

\author{Nik Weaver}

\title [The semantic conception of proof]
       {The semantic conception of proof}

\address {Department of Mathematics\\
          Washington University in Saint Louis\\
          Saint Louis, MO 63130}

\email {nweaver@math.wustl.edu}

\date{\em December 12, 2013}

\maketitle


\section{}

Classically, the notions of truth and provability are both directly
tied to the meanings of the logical constants. The typical setting
involves a model $M$ and a language $L$
which is to be interpreted in $M$. The informational content of $L$
is determined by the way we specify which sentences of $L$ are true.
Fixing an interpretation, the usual way to do this is to begin by
defining a function from the atomic formulas of $L$ into the set
$\{{\rm true},{\rm false}\}$ in some way that reflects the structure
of $M$. We then extend this function to complex formulas recursively
by means of truth conditions (e.g., $A \wedge B$ is true iff $A$ is
true and $B$ is true). This is the classical account of the role of
truth in specifying the meanings of the logical constants.

We also have a notion of deductive reasoning as the means by which we
come to know which sentences are true. Proofs typically proceed though
a series of inferences which conform to certain deduction rules.
These rules can be formulated in terms very similar to the truth conditions,
in a way that also transparently reflects the meanings
of the logical constants (e.g., given $A$ and $B$ one may infer $A \wedge B$,
and given $A \wedge B$ one may infer both $A$ and $B$) \cite{Pra}. It is
easy to see that deductions preserve truth, so that any proof which proceeds
from axioms which are true in the model will establish a conclusion which
is also true in the model. Thus anything which is provable should be true,
but there is no reason to expect the converse to hold in general.

However, this implication (provable implies true) can only be affirmed for
formal proofs relative to some system of axioms which are known to be
true. It is a weakness of the classical picture that we do not have a
general account of how we are to go about choosing the axioms on which
we base our proofs. The seriousness of this problem can already be seen in
the case of first order number theory, where it is generally accepted that
the usual Peano axioms, for instance, do not exhaust our knowledge of what
is true in the model. Indeed, because the true sentences of first order
arithmetic are not recursively enumerable, we will always have the capacity
to strengthen any recursively presented formal system for number theory
that is known to be sound by adding a (standardly expressed arithmetical)
statement of its consistency. This consistency statement cannot have been
formally provable in the original system, but if we can see that that
system is sound then its consistency is certainly provable in an informal
semantic sense. And if we do not recognize that the original system is sound,
then the formal proofs of that system clearly cannot function as genuine
proofs in a semantic sense. So in neither case do the formal proofs of
the given system adequately express the semantic concept of a valid proof.
These considerations suggest that the semantic notion of a valid proof as
a linguistic object which compels rational assent is not so easy to formally
capture.

The subtlety of the semantic notion of a valid proof, in contrast to the
straightforwardness of the syntactic notion of a valid proof within a formal
system, might lead us to focus on the latter at the expense of the former.
But we should not lose sight of the fact that the main reason we care about
formal proofs is because we think they exhibit semantic validity.

\section{}

The classical picture outlined above is unsuited to cases where we
are not able to affirm that every atomic formula of the language has a
well-defined truth value in the model (relative to a given interpretation).
There are various ways this can happen. For instance, the recursive
enumerability issue mentioned above makes it doubtful that the assertion
``$A$ is provable'', in the sense of semantic provability (and restricting
to proofs of finite length), can be said to have a definite truth value for
every sentence $A$ of first order number theory.

Actually, once we are in the business of explicitly reasoning about semantic
provability, there is no need to invoke G\"odelian incompleteness to make
this point. It is easy to see directly how assigning a truth value to a
sentence that asserts the soundness of some deduction rule could be
problematic. For in order to do this we would need to know whether the
rule preserves truth, but that would depend on which sentences are true,
and could potentially hinge on the truth value of the sentence under
consideration. A vicious circle is possible.

Thus, when we are reasoning about provability, we cannot necessarily assume
that every atomic proposition has a definite truth value. This creates
difficulties for the interpretation of the logical constants; we can no
longer characterize them in terms of truth conditions. The constructivist
solution to this sort of problem is to take the deduction rules, rather
than the truth conditions, as our starting point. This is possible because
the deduction rules, despite intuitively having essentially the same content
as the truth conditions, are nonetheless formulated without reference
to truth values. Thus, classically, we use truth
conditions to give meaning to the logical constants, and we justify the
deduction rules by observing that they preserve truth, but constructively,
we discard the truth conditions, regard provability as primary, and use
the deduction rules to give meaning to the logical constants. The rules
for $\wedge$ tell us what $\wedge$ means in the sense that they tell us
under what circumstances we are entitled to assert $A \wedge B$: we
are entitled to assert $A \wedge B$ precisely if we are entitled to
assert both $A$ and $B$. This approach gives us a way to reason about
formulas that might not have definite truth values.

An unfortunate byproduct of this emphasis on provability is that
constructivists are not always careful about distinguishing between a
statement and the assertion that that statement is provable.\footnote{E.g.,
``The assertion of $A \vee \neg A$ is therefore a claim to have, or to be
able to find, a proof or disproof of $A$'' (\cite{Dum}, p.\ 21). More
correct would be ``We are entitled to assert $A \vee \neg A$ if we have, or
are able to find, a proof or disproof of $A$.''} The two are not synonymous.
That we can prove there are infinitely many primes is indeed what licenses
us to assert there are infinitely many primes, and we can arguably regard the
meaning of the assertion that there are infinitely many primes as residing
in our characterization of what constitutes a proof of this assertion, but
saying there are infinitely many primes is not the same as saying
we can prove there are infinitely many primes. One is a statement about
numbers, the other a statement about proofs. For our purposes here,
preserving this distinction is absolutely crucial.

\section{}

We take a proof to be a syntactic object --- not a ``procedure'' or a
``construction'' --- that compels rational acceptance of some conclusion.
We consider the proof relation to be a primitive notion that cannot be
defined in any simpler terms. The most basic feature of a valid proof, in
this conception, is that it must be recognizably valid. That is, if $p$
proves $A$ then we must in principle be able to recognize that $p$ proves
$A$. This is inherent in what we mean by a proof. If we cannot see that
$p$ compels us to accept $A$, then we do not consider $p$ to be a proof
of $A$ in the semantic sense of concern to us here.

It is not important for us whether proofs are expressed linearly or
as trees (or in some other way), just that they be syntactic
objects. We assume only that we have a notion of one proof being contained
within another and that for any two proofs there is a third that contains
both of them. Also, we will allow proofs to contain extraneous material. It
is not clear how a prohibition on inessential content could be formulated,
but even if it could, doing so would be inconvenient for us. Therefore we
stipulate that if $p$ proves $A$ then $p'$ also proves $A$, for any proof
$p'$ containing $p$.

We now want to work out how the meanings of the various logical constants
can be expressed in terms of the proof relation. For instance, we take a
proof of $A \vee B$ to be something
which is either a proof of $A$ or a proof of $B$. Thus $p$ proves $A \vee B$
if and only if $p$ proves $A$ or $p$ proves $B$. Similarly, a proof of
$A \wedge B$ is something which is both a proof of $A$ and a proof of $B$.
(Allowing proofs to contain extraneous material simplifies things here; a
proof of $A \wedge B$ does not have to be a pair of proofs, it can simply
be a proof that contains both a proof of $A$ and a proof of $B$.) So $p$
proves $A \wedge B$ if and only if $p$ proves $A$ and $p$ proves $B$.

A proof of $(\exists x)A(x)$ must describe a construction, prove that
the construction produces a unique object $a$, and prove the statement
$A(a)$. Since it is a name for an object, not the object itself, that
appears in the assertion $A$, in order to give a proof interpretation
of existential quantification we need to require that every individual
in the model be represented by a term in the language. This can always
be assured by simply including a constant symbol in the language for
each individual in the model. (Cf.\ the need to represent abstract
objects by concrete proxies, discussed in Section 3 of \cite{W0}.)

Our accounts of $(\forall x)A(x)$ and $A \to B$ require a more detailed
explanation. In the traditional literature their proof interpretations
ask for a construction which, for each object $a$, produces a proof
of $A(a)$, or a construction which
converts any proof of $A$ into a proof of $B$. This kind of formulation
is implausible on its face. For example, consider a computer program that
inputs natural numbers $a$, $b$, $c$, and $n$, and
then, if $n \geq 3$, prints out a step-by-step evaluation of both
$a^n + b^n$ and $c^n$ and a
verification that they are unequal (say, by comparing them digit by digit).
Since Fermat's last theorem is true, it follows that this procedure does
in fact convert any quadruple $(a,b,c,n)$ such that
$n \geq 3$ into a proof that $a^n + b^n
\neq c^n$. So according to the definition just mentioned,
this trivial procedure would count as a proof of Fermat's last theorem.

We might consider adding a clause to the effect that the procedure must
not only produce a proof of $A(a)$ for each $a$, but also be
recognized to do so. (In fact, it is clear that we have to do this if we
are to sustain our requirement that any proof can be recognized to be a
proof.) The problem here is that recognizing that it is possible to
generate a proof of $A(a)$ is not the same as being compelled to
accept $A(a)$. It is, rather, the same as being compelled to accept that
$A(a)$ is provable. This is just the distinction we emphasized earlier
between asserting some statement and asserting that that statement is
provable. There really is no way to maintain it here because it is essential
to the construction viewpoint that one does not actually have a proof of
$A(a)$ for every $a$, one only has a way of generating these
proofs.

This issue might be confusing because of the condition that any valid
proof must be recognizably valid. It follows from this condition that if
$p$ proves $A$ then it not only compels acceptance of $A$, it must also
compel acceptance of the fact that it proves $A$. This means that any proof
of $(\forall x)A(x)$ would actually also prove that $A(a)$ is provable
for every $a$. But it is the converse implication (going from
``$A(a)$ is provable'' to $A(a)$) that would be needed to
justify the construction interpretation.

The only thing we can do is to define a proof of $(\forall x)A(x)$ to be a
single argument $p$ which, for every object $a$, compels us to accept
the statement $A(a)$. Now, it is natural to object that demanding a
uniform proof of $A(a)$ is too restrictive. The concern is that it
might be possible to, in some uniform way, generate proofs of $A(a)$
for various values of $a$ without being able to give a single
proof that simultaneously covers all cases. But this objection is not
well-taken. Not only does it, as we have just explained, ignore the
distinction between proving $A(a)$ and proving that $A(a)$ is
provable, it also fails on its own terms. For suppose we had a construction
that produced a proof of $A(a)$ for every $a$. How could we know
that it did this? If there are a proper class of possible values of $a$
then direct inspection is not an option. But asking the construction
to be accompanied by a collection of proofs, one for each $a$, that it
produces a proof of $A(a)$ would be absurd; if we allowed that then we
could just as well discard the construction altogether and simply ask for
a collection of proofs of $A(a)$, one for each $a$. On the other
hand, demanding a single uniform proof that the construction produces
a proof of $A(a)$ for every $a$ negates the original basis of the
objection. Any approach at some point has to come down to a single
proof; all the construction viewpoint accomplishes is to erase the distinction
between a statement and the assertion that that statement is provable.

Similar comments apply to implication. Again, to prove $A \to B$ it is not
enough to know that we can use any proof of $A$ to generate a proof of $B$; if
anything, that would establish that $A$ is provable implies $B$ is provable,
not that $A$ implies $B$. The proper formulation is that a proof of $A \to B$
is an argument that, granting $A$, compels rational acceptance of $B$. We
can say that $p$ proves $A \to B$ if and only if every $p' \supseteq p$
that proves $A$ also proves $B$.

Now that we have clarified universal quantification, we can say a little
more about existential quantification. Consider statements of the form
$(\forall x)(\exists y)A(x,y)$. By what we just said, any proof of such a
statement must be a single argument $p$ which, for every object $a$,
proves $(\exists y)A(a,y)$. If $a$ ranges over a proper class of values
then we cannot expect $p$ to explicitly present the corresponding
$b$ in every case. Here we must accept the construction point of
view and ask that $p$ merely tell how to generate $b$. Of course
it must also prove that the generating procedure works, and that
$A(a,b)$ holds.

This completes our informal explication of the logical constants in terms of
provability. Negation is handled by taking $\neg A$ to be an abbreviation
of $A \to \bot$, where $\bot$ represents some canonical falsehood.

We can also reason about the proof relation itself. The idea that proofs
are recognizable implies that if $p$ proves $A$ then the empty proof proves
that $p$ proves $A$. In the reverse direction, it is certainly the
case that whenever we have proven a statement we accept that statement.
However, we cannot take this implication as an axiom because of the
possible circularity involved in its application to proofs in which it
appears. It can only be formulated as a deduction rule: given that $p$
proves $A$, infer $A$. We will elaborate on this point below.

The proof interpretation of the logical constants presented in this
section immediately justifies the usual axioms and rules of minimal
first order predicate calculus. (In particular, the empty proof proves
every tautology.) Thus we are now free to reason accordingly. To recover
intuitionistic logic in some setting we would need to give a special
justification for the ex falso quodlibet law, and to recover classical
logic we would also need to give a special justification for the law
of excluded middle.

\section{}

Write $p\,\vdash A$ for ``$p$ proves $A$''.
To assert that $A$ is provable is to assert that there exists a proof
of $A$. Writing $\Box A$ for ``$A$ is provable'', we therefore have
$\Box A \equiv (\exists p)(p \vdash A)$.
Basic laws relating provability to the logical constants can now
be derived. For instance, a proof of $\Box(A \vee B)$ would
have to construct some $p$ and prove that $p \vdash (A \vee B)$. But
this means it would either prove $p \vdash A$ or $p \vdash B$, and
in either case it would then be a proof of $\Box A \vee \Box B$.
Thus we see that $\Box(A \vee B)$ implies $\Box A \vee \Box B$ in general.
Conversely, any proof of $\Box A \vee \Box B$ is either a proof of
$\Box A$ or $\Box B$, so that it constructs some $p$ and proves either
$p \vdash A$ or $p \vdash B$, and in either case is qualifies as a proof
of $\Box(A \vee B)$. We conclude that $\Box(A \vee B)$ and $\Box A \vee
\Box B$ are equivalent. The equivalence of $\Box(A \wedge B)$ and
$\Box A \wedge \Box B$ can be seen in a similar way, this time using
the property that if $p$ proves $A$ and $q$ proves $B$, then any proof
containing both $p$ and $q$ proves both $A$ and $B$.

A proof of $\Box(\forall x)A(x)$ constructs $p$ and proves $p \vdash A(a)$
for every object $a$. Therefore, for each $a$ it proves
$\Box A(a)$. Thus $\Box(\forall x)A(x) \to (\forall x)\Box A(x)$.
The converse direction is not evident because having a separate proof
of $A(a)$ for each $a$ should not entail that there is a single
proof of $A(a)$ for all $a$. An exception can be made if proofs are
allowed to be as large as the range of values of $a$: for instance, if
$a$ ranges over a finite set and proofs can be any finite size, then
proofs for each $a$ actually could be amalgamated into a single proof
of $(\forall x)A(x)$. Or if the range of values of $a$ is countable and
we allow proofs of countable size, then again such an amalgamation is possible.
A natural choice is to let proofs be as large as any set; then the
law $(\forall x)\Box A(x) \to \Box(\forall x)A(x)$ would hold if $x$
ranges over a set of values, but not if it ranges over a proper class
of values.

A proof of $(\exists x)\Box A(x)$ would construct $a$ and $p$ and prove
$p \vdash A(a)$. So it would prove $p \vdash (\exists x)A(x)$; thus
$(\exists x)\Box A(x) \to \Box(\exists x)A(x)$. Conversely, a proof of
$\Box(\exists x) A(x)$ would construct $p$ and prove
$p \vdash (\exists x)A(x)$.
We would then like to extract an object $a$ such that $p \vdash A(a)$, but
that assumes the soundness of $p$, so we cannot draw this inference.
Again, we will elaborate on this point below.

For implication, a proof of $\Box(A \to B)$ identifies $p$ and proves
$p\vdash(A \to B)$, i.e., it proves that $p'$ proves $A$ implies $p'$ proves
$B$ for any proof $p'$ containing $p$. Putting this together with a proof
of $\Box A$, i.e., a construction of $q$ such that $q \vdash A$,
yields in a uniform way (by adding the instruction to combine $p$ and $q$)
a proof of $\Box B$. So we have shown that $\Box(A \to B)$
implies $\Box A \to \Box B$. The converse is not evident.

We also consider the general relationship between $A$ and $\Box A$. Recall
that if $p$ proves $A$ then $\emptyset$ proves that $p$ proves $A$; this
yields that the empty proof is always a proof of $A \to \Box A$. The converse
is not evident. For suppose we have a proof that $A$ is provable, i.e.,
a construction of $p$ and a proof that $p$ proves $A$. To extract a proof of
$A$ from this premise we would need to rely on the inference from having a
proof of $p \vdash A$ to simply having $p \vdash A$. So we would have to
use the law $\Box A \to A$ which we are trying to prove. There is no obvious
way around this circularity. We can only affirm the deduction rule:
given $\Box A$, infer $A$.

To summarize, we have the axioms
\begin{eqnarray*}
\Box(A \vee B) \qquad&\leftrightarrow&\qquad \Box A \vee\Box B\cr
\Box(A \wedge B) \qquad&\leftrightarrow&\qquad \Box A \wedge \Box B\cr
\Box(\exists x)A \qquad&\leftarrow&\qquad (\exists x)\Box A\cr
\Box(\forall x)A \qquad&\rightarrow&\qquad (\forall x)\Box A\cr
\Box(A \to B) \qquad&\to&\qquad (\Box A \to \Box B)\cr
A \qquad&\to&\qquad \Box A
\end{eqnarray*}
and the deduction rule that infers $A$ from $\Box A$, with the
implication $(\forall x)\Box A \to \Box(\forall x)A$ also holding
if there is a suitable size restriction on the range of values of $x$.

\section{}

The most important conclusion we reached in the last section is that we do
not have a right to affirm the law $\Box A \to A$ in general. This may be
counterintuitive because it seems as though having a proof that there is a
proof of $A$ should be just as good as having a proof of $A$. Once we
have accepted a line of reasoning which establishes that
$p$ proves $A$ we ought to then accept $p$ as a proof of $A$. The problem
is the circularity that arises when we try to affirm this inference as a
general principle. It becomes circular when it is adopted as a universal
law because it then has the effect of affirming the soundness of proofs
in which it might itself have been used.

The situation is analogous to the difficulty associated with formal
systems that prove their own consistency. Once we have accepted a formal
system ${\mathcal S}$ we do generally agree to accept a stronger system
${\mathcal S}'$ obtained by augmenting ${\mathcal S}$ with a (standardly
expressed arithmetical) assertion of the consistency of ${\mathcal S}$.
However, the new consistency axiom is only applied in retrospect, to
affirm the correctness of reasoning which could be executed in the
original system ${\mathcal S}$. We should not accept a new axiom
which expresses the consistency not of the original system, but
of the new system formed by adjoining that axiom itself. That would be
circular, and we know from G\"odel's second incompleteness theorem that
it is not a benign sort of circularity; any such axiom would have to give
rise to an inconsistency. In just the same way, once we accept some formal
system we should agree, whenever that system proves that $A$ is provable,
to accept $A$; but we should not agree to vouch for the original
system augmented by the axiom $\Box A \to A$. Adding the latter axiom would
affirm the soundness not just of all proofs in the original system, but also
of all proofs in the augmented system, including proofs that employ the new
axiom itself.

Maintaining a distinction between $A$ and $\Box A$ is essential for properly
handling the semantic paradoxes involving provability \cite{W1}. Most
importantly,
given that some statement $L$ entails that a contradiction is provable, i.e.,
$L \to \Box \bot$, we cannot infer that $L$ is false, i.e., $L \to \bot$.
Since we do have $\bot \to \Box \bot$ as a special case of the law
$A \to \Box A$, an inference can be drawn in the opposite direction: given
$L \to \bot$ we may deduce $L \to \Box \bot$. In this sense the assertion
$L \to \bot$, which is the standard negation of $L$, is stronger than the
assertion $L \to \Box \bot$. We call $L \to \Box \bot$ the {\it weak
negation} of $L$ and when we have proven this we say that $L$ is
{\it weakly false}.

The semantic paradoxes involving provability collapse when we are careful
to distinguish a statement from the assertion that that statement is
provable. For instance, let $L$ be a sentence which asserts that its
negation is provable. Assuming $L$, we can then immediately
infer $\Box \neg L$. We can also infer $\Box L$ from $L$,
as an instance of the general law $A \to \Box A$. Putting these together
yields $L \to \Box \bot$, i.e., $L$ is weakly false. Taking $\neg L$ as
a premise instead allows us to infer $\Box \neg L$, which is equivalent
to $L$, and this shows that $\neg L$ is false. So the conclusion we reach
is that $L$ is weakly false and $\neg L$ is false. But there is no
contradiction here.

In this example the distinction between falsity and weak falsity is crucial.
If the two were equivalent then we would have proven both $\neg L$ and
$\neg\neg L$, which is absurd. What this means is that it is not merely
the case that we cannot affirm the implication $\Box \bot \to \bot$. Since
this implication would imply that falsity is equivalent to weak falsity, it
would, together with the preceding analysis of $L$, yield a contradiction.
In other words, we have shown $(\Box \bot \to \bot) \to \bot$, which can
also be written as $\neg(\Box \bot \to \bot)$ or as $\neg\neg \Box \bot$.
Thus, in the case where $A$ equals $\bot$, we can affirm that the law
$\Box A \to A$ is false.

This conclusion would be unacceptable if we were to interpret implication
classically. According to the classical interpretation, the only way an
implication can fail is if the premise is true and the conclusion is false;
so the law $\Box \bot \to \bot$ could only fail if $\bot$ were actually
provable. But then, if we were reasoning classically we could apply the
law of excluded middle to $L$ and infer $\Box \bot$ from the fact that
both $L$ and $\neg L$ are weakly false. So in fact we could actually prove
$\Box \bot$.

This only confirms our earlier judgement that it is
inappropriate to reason about provability classically. Indeed, liar type
sentences are just the sort of thing that show us we may not always be in
a position to assign definite truth values to assertions about provability.
So we have to reason constructively, and affirming that $\Box \bot \to \bot$
is false is perfectly reasonable under a constructive interpretation of
implication. It merely expresses the idea that if we could find some way
of converting any (hypothetical) proof that a falsehood is provable into
a proof of a falsehood, then a contradiction would result. In simpler
language, assuming that $\bot$ is unprovable leads to a contradiction. As
we explained above, this assumption is indeed unwarranted because it hinges
on the global soundness of all proofs, and any proof whose soundness is
justified by invoking the global soundness of all proofs is a proof whose
justification is circular.

Here too there is an analogy with the G\"odelian analysis of consistency
in number theoretic systems. We can prove, in Peano arithmetic, that if
PA proves ${\rm Con}($PA$)$ then PA proves $0 = 1$. Since ${\rm Con}($PA$)$
is just the statement that $0 = 1$ is not provable in PA, this amounts to
saying that assuming PA proves that $0 = 1$ is unprovable in PA leads to a
contradiction in PA. This mirrors our conclusion that assuming
$\bot$ is unprovable leads to a contradiction.

\section{}

In the last section we showed how a certain liar type sentence narrowly
avoids paradox. This happens because we lack both the law $\Box A \to A$
(otherwise we could prove $\bot$) and the law of excluded middle (otherwise
we could prove $\Box\bot$, and then infer $\bot$). We now want to prove in
a formal setting that the kind of reasoning exhibited above actually is
consistent, even when dealing with circular phenomena.

We formulate a propositional system ${\mathcal P}$ for reasoning about
assertions which reference each other's provability in a possibly
circular way. The atomic formulas of the language consist of the
falsehood symbol $\bot$ together with finitely many propositional
variables $\phi_1$, $\ldots$, $\phi_n$. Complex formulas are generated
from the atomic formulas by the rule that whenever $A$ and $B$ are
formulas, so are $A \wedge B$, $A \vee B$, $A \to B$, and $\Box A$.

The logical axioms and rules of ${\mathcal P}$ are those of a minimal
propositional calculus, together with the axioms and rule for $\Box$
presented in Section 4 (minus the two axioms for quantification).

The nonlogical axioms of ${\mathcal P}$ will consist of a list of formulas
$\phi_1 \leftrightarrow A_1$, $\ldots$, $\phi_n \leftrightarrow A_n$,
where each $A_i$ can be any formula in which every propositional variable
lies within the scope of some box operator. Thus
$\phi_1 \leftrightarrow \neg \Box \phi_1$ and
$\phi_2 \leftrightarrow \Box \neg\phi_2$ are acceptable axioms but
$\phi_3 \leftrightarrow \neg \phi_3$ is not. The list might include
liar pairs such as $\phi_4 \leftrightarrow \Box \phi_5$ and $\phi_5
\leftrightarrow \neg\Box\phi_4$. The restriction on the $A_i$'s
arises from the grammatical distinction between asserting a proposition
and merely mentioning its name. ``The negation of $\phi_3$'' is not a
cogent assertion; we have to say something like ``$\phi_1$ is not
provable'' or ``the negation of $\phi_2$ is provable''.

It is easy to show that ${\mathcal P}$ minus the rule which infers
$A$ from $\Box A$ is consistent: just make $\Box A$ true for every
formula $A$ and evaluate the truth of the $\phi_i$ and all remaining
formulas classically; then the set of true formulas contains all the
theorems of ${\mathcal P}$ but does not contain $\perp$. Proving that
$\Box^k \bot$ is not a theorem of this system for any $k$ is a little
more substantial.

\begin{theo}
${\mathcal P}$ is consistent.
\end{theo}

\begin{proof}
We define the level $l(A)$ of a formula $A$ of ${\mathcal P}$ as
follows. The level of $\bot$ and any formula of the form $\Box A$
is 1. The level of $A \wedge B$, $A \vee B$, and $A \to B$ is
$\max(l(A), l(B)) + 1$. The level of $\phi_i$ is $l(A_i) + 1$.

Now we define a sequence of sets of formulas $F_1$, $F_2$, $\ldots$.
These can be thought of as the formulas which we have
determined not to accept as true. The definition of $F_k$ proceeds by
induction on level. When $k = 1$, the formula $\bot$ belongs to $F_1$
but no other formula of level 1 belongs to $F_1$. For $k > 1$, $\bot$
belongs to $F_k$ and $\Box A$ belongs to $F_k$ for every formula $A$
which belongs to $F_{k-1}$. For levels higher than 1, we apply the
following rules. If either $A$ or $B$ belongs to $F_k$ then $A \wedge B$
belongs to $F_k$. If both $A$ and $B$ belong to $F_k$ then $A \vee B$
belongs to $F_k$. If $A_i$ belongs to $F_k$ then $\phi_i$ belongs to
$F_k$. Finally, if there exists $j \leq k$ such that $B$ belongs to $F_j$
but $A$ does not belong to $F_j$, then $A \to B$ belongs to $F_k$.

An easy induction shows that the sequence $(F_k)$ is increasing.
We define $F = \bigcup F_k$. It is obvious that $\bot$
belongs to $F$. The proof is completed by checking that
none of the axioms of ${\mathcal P}$ belongs to $F$, and that the
complement of $F$ is stable under modus ponens and the inference of
$A$ from $\Box A$. This is tedious but straightforward.
\end{proof}

An identical argument would prove the consistency of a
stronger system with the box axiom for implication strengthened to
$\Box(A \to B) \leftrightarrow (\Box A \to \Box B)$ and minimal
logic strengthened to intuitionistic logic. However, the justification
for these stronger axioms is unclear. In particular, there is no
obvious way of defining $\bot$ so as to justify the ex falso law in
this setting. Since $\bot$ could appear in the defining formulas for
the $\phi_i$, a circularity issue arises if we try to build the
implications $\bot \to \phi_i$ into the definition of $\bot$.

\section{}
Taking $\phi \leftrightarrow \neg \phi$ as an axiom guarantees
inconsistency, but as we have just seen, $\phi \leftrightarrow
\neg\Box \phi$ does not. Thus, inserting box operators into the
axioms of an inconsistent theory can sometimes make it consistent.
We now want to describe a general technique for proving results of
the opposite type which state that in some cases inserting box
operators into axioms does not essentially weaken them.

Let ${\mathcal T}_1$ and ${\mathcal T}_2$ be formal systems based
on the same underlying logic (classical, intuitionistic, or minimal), such
that the language of ${\mathcal T}_2$ is the language of ${\mathcal T}_1$
augmented with $\Box$. Also assume that the axioms and deduction rule
for $\Box$ are
included in the theory ${\mathcal T}_2$. Then we say that ${\mathcal T}_2$
{\it weakly interprets} ${\mathcal T}_1$ if, when all appearances of
$\Box$ are deleted from the theorems of ${\mathcal T}_2$, the resulting
set of formulas contains all the theorems of ${\mathcal T}_1$. Informally,
${\mathcal T}_2$ weakly interprets ${\mathcal T}_1$ if ${\mathcal T}_2$
interprets ${\mathcal T}_1$ when we ignore the difference between $A$
and $\Box A$. We have the following trivial result:

\begin{prop}
If ${\mathcal T}_2$ weakly interprets ${\mathcal T}_1$ then the
consistency of ${\mathcal T}_2$ implies the consistency of ${\mathcal T}_1$.
\end{prop}

This is just because if ${\mathcal T}_1$ were inconsistent then $\bot$
would be a theorem of ${\mathcal T}_1$, and weak interpretability then
implies that $\Box^k \bot$ must be a theorem of ${\mathcal T}_2$ for some
$k$. Repeated application of the rule ``infer $A$ from $\Box A$'' then
shows that $\bot$ is a theorem of ${\mathcal T}_2$.

We will now describe a procedure for inserting box operators into formulas
that lack them and prove that if the initially given formula is a theorem
of a standard (classical, intuitionistic, or minimal) predicate calculus
then we can ensure that the formula generated by our procedure will be
a theorem of that predicate calculus augmented by the axioms for $\Box$.
This result can be used to establish weak interpretability results; we
illustrate this in the corollary below.

Our procedure can be described as a game between two players, Attacker
and Defender, on a formula $A$. We think of Attacker as seeking to
strengthen the formula and Defender as seeking to weaken it.
The way the game is played is defined inductively on
the complexity of $A$. If $A$ is atomic then the game consists of a
single move in which Defender chooses a value of $k$ (possibly zero)
and prefixes $A$ with $\Box^k$. Attacker does not have a turn.
The game is played on formulas of the form $A \wedge B$ by independently
playing the game on $A$ and $B$ and conjoining the results. It is played
on $A \vee B$ by independently playing on $A$ and $B$ and disjoining the
results. It is played on $(\exists x)A$ and $(\forall x)A$ by playing on
$A$ and prefixing the result with the relevant quantifier. Finally, it
is played on $A \to B$ by first having the players switch roles and play
on $A$, producing a formula $A'$, then revert to their original roles and
play on $B$, producing a formula $B'$. The result of this game is the
formula $A' \to B'$.

If the initially given formula is a theorem of a standard predicate calculus,
then Defender wins provided the formula generated by the game is a theorem
of the same standard predicate calculus augmented by the axioms for $\Box$.

\begin{theo}
Defender has a winning strategy on any theorem of a standard predicate
calculus.
\end{theo}

\begin{proof}
All moves in the game consist in one of the players choosing a value of $k$
and prefixing an atomic formula with $\Box^k$ for some $k$. Thus a strategy
for either player is given by specifying, for each of his moves, the
value of $k$ to be played as a function of the values chosen by his
opponent on all of the opponent's earlier moves.

We order strategies by saying that $S \leq S'$ if, at every play, for each
choice of earlier moves by the opponent, the value of $k$ prescribed by
$S'$ is greater than or equal to the value prescribed by $S$. We claim
that if Attacker plays the same moves against two of Defender's strategies,
$S$ and $S'$, such that $S \leq S'$, then the formula generated by the
first game will imply the formula generated by the second game, and if
Defender plays the same moves against two of Attacker's strategies,
$T$ and $T'$, such that $T \leq T'$, then the formula generated by the
first game will be implied by the formula generated by the second game.
This is shown by a straightforward induction on the complexity of the
formula on which the game is played, taking both claims for all simpler
formulas as the induction hypothesis. It follows that any strategy
greater than a winning strategy also wins.

The theorem can be proven using Gentzen-style sequent calculus. We work
with the G1 systems of \cite{TS}, using only atomic formulas in the axioms
Ax and (in the intuitionistic and classical cases) replacing the axiom
L$\bot$ with the axioms $\bot \Rightarrow A$ for all atomic formulas $A$.
We first extend the definition of the game so that it can be played on
sequents. On a sequent of the form $\Gamma, A \Rightarrow B, \Delta$
where $\Gamma$ and $\Delta$ are the context, the game is played by
having Attacker and Defender switch roles and play the game on the
formulas of $\Gamma$ in any order and then on $A$, then revert to their
original roles and play the game on $B$ and then, finally, on the formulas
of $\Delta$ in any order. The proof is completed by checking that
Defender has a winning strategy on any axiom, and if Defender has a
winning strategy on the premises of a rule then he has a winning
strategy on the conclusion of that rule. This is straightforward but
tedious. In every case the strategy adopted by Defender on each
formula in the conclusion of a rule will be the strategy he used
on the same formula in one of the premises of that rule. Since
increasing a winning strategy always produces a winning strategy,
if a formula appears in more than one premise we can assume that
Defender played the same strategy in both cases.
\end{proof}

Say that a formula is {\it increasing} if no implication appears in
the premise of any other implication. Note that since we take
$\neg A$ to be an abbreviation of $A \to \bot$, this also means that
an increasing formula cannot position a negation within the premise
of any implication, nor can it contain the negation of any implication.

\begin{coro}
Suppose the nonlogical axioms of ${\mathcal T}_2$ are increasing and
the nonlogical axioms of ${\mathcal T}_1$ are those of ${\mathcal T}_2$
with all boxes deleted. Then ${\mathcal T}_2$ weakly interprets
${\mathcal T}_1$.
\end{coro}

\begin{proof}
Let $B$ be any theorem of ${\mathcal T}_1$. We must find a way to
insert boxes into $B$ so that it becomes a theorem of ${\mathcal T}_2$.
Since ${\mathcal T}_1$ proves $B$, it is a logical consequence of finitely
many nonlogical axioms $A_i$ of ${\mathcal T}_1$; writing their conjunction
as $A = \bigwedge A_i$ we then have that $A \to B$ is a theorem
of a standard predicate calculus. The problem is
to insert boxes into $A$ and $B$, yielding new formulas $A'$ and
$B'$, in such a way that $A' \to B'$ is a theorem of standard predicate
calculus augmented by the axioms for $\Box$, and such that
each conjunct $A_i'$ of $A'$ is implied by the corresponding
nonlogical axiom $A_i''$ of ${\mathcal T}_2$. It will then follow
that $B'$ is a theorem of ${\mathcal T}_2$, as desired.

We achieve this result by playing the game described above on the formula
$A \to B$. We know that Defender can ensure that $A' \to B'$ is a
theorem of standard predicate calculus augmented by the axioms for
$\Box$, so all we need to do is to
prescribe a strategy for Attacker which ensures that each conjunct
$A_i'$ of $A'$ is implied by the axiom $A_i''$.

The game is played on $A \to B$ by first switching the players' roles
and playing on $A$. So we reduce to a problem about finding a strategy for
Defender when the game is played on $A$, and this amounts to giving a
strategy for Defender on each conjunct $A_i$. We know that the formula
$A_i''$ is obtained from $A_i$ by inserting boxes in some way, and we
need to provide Defender with a strategy for playing the game on $A_i$
in a way that ensures $A_i'' \to A_i'$. We prove this can be achieved
for any increasing formula $A_i$ recursively on its complexity. The only
interesting case is when $A_i$ is an implication, $A_i = A_{i0} \to A_{i1}$.
We require a lemma which states that if a formula $C$ contains no implications
and $C'$ is obtained from $C$ by inserting boxes in some way, then
$C \to C' \to \Box^j C$ (i.e., $C \to C'$ and $C' \to \Box^j C$) for some
value of $j$. This is easily shown by induction on the complexity of $C$.
This lemma can be applied to $A_{i0}$ because $A_i$ is increasing, so no
matter how the game is played on $A_{i0}$ we will have
$A_{i0}' \to \Box^j A_{i0} \to \Box^jA_{i0}''$ for some $j$. Inductively
Defender can then employ a strategy on $A_{i1}$ which ensures
that $\Box^j A_{i1}'' \to A_{i1}'$, and this yields that
$A_{i0}'' \to A_{i1}''$ implies
$$A_{i0}' \to \Box^j A_{i0}'' \to \Box^j A_{i1}'' \to A_{i1}',$$
that is, $A_i''$ implies $A_i'$, as desired.
\end{proof}


\end{document}